\documentclass{article}
\usepackage{amsthm,amsmath,amsfonts,vmargin}
\usepackage{hyperref}

\setpapersize{A4}
\setmargrb{30mm}{20mm}{30mm}{20mm}

\def\N{\mathbb{N}}
\def\Z{\mathbb{Z}}
\def\F{\mathbb{F}}

\def\EEE{\left\langle\left\langle E \right\rangle\right\rangle}
\def\spA{\left\langle A \right\rangle}

\def\IS{\mathrm{IS}}
\def\height{\mathrm{height}}
\def\span{\mathrm{span}}
\def\codim{\mathrm{codim}}

\newtheorem{thm}{Theorem}
\newtheorem*{thm*}{Theorem}
\newtheorem{claim}{Claim}
\newtheorem*{claim*}{Claim}
\newtheorem{lemma}[thm]{Lemma}
\newtheorem*{lemma*}{Lemma}

\newtheorem*{prop*}{Proposition}
\newtheorem{cor}[thm]{Corollary}
\newtheorem*{cor*}{Corollary}
\newtheorem{conj}[thm]{Conjecture}
\newtheorem*{conj*}{Conjecture}
\theoremstyle{remark}

\newtheorem*{rmk*}{Remark}

\newtheorem*{exm*}{Example}

\newtheorem*{rmks*}{Remarks}

\title{The Freiman--Ruzsa Theorem over Finite Fields}

\author{
Chaim Even-Zohar
\thanks{
Einstein Institute of Mathematics, HUJI, Israel. e-mail: \texttt{chaim.evenzohar@mail.huji.ac.il}.}
\and
Shachar Lovett
\thanks{CSE department, UC San-Diego. e-mail: \texttt{slovett@cse.ucsd.edu}. Supported by NSF CAREER award 1350481.}
}

\begin{document}

\maketitle

\begin{abstract}
Let $G$ be a finite abelian group of torsion $r$ and let $A$ be a subset of $G$.
The Freiman--Ruzsa theorem asserts that if $|A+A| \le K|A|$
then $A$ is contained in a coset of a subgroup of $G$ of size at most $K^2 r^{K^4} |A|$.
It was conjectured by Ruzsa that the subgroup size can be reduced to $r^{CK}|A|$
for some absolute constant $C \geq 2$.
This conjecture was verified for $r=2$ in a sequence of recent works,
which have, in fact, yielded a tight bound.
In this work, we establish the same conjecture for any prime torsion.

\begin{flushleft}
\emph{Keywords:} sumset; abelian; doubling; compression

\emph{MSC:} 11P70 (05D99)
\end{flushleft}

\end{abstract}

\section{Introduction}

Let $A$ be a subset of a finite abelian group.
The \emph{doubling constant} of $A$ is defined by $|A+A|/|A|$,
where as usual $A+B = \{a+b \;|\; a \in A, b \in B \}$.
The \emph{spanning constant} of $A$ is defined by $|\left\langle A \right\rangle|/|A|$,
where $\left\langle A \right\rangle$ is the \emph{affine span} of $A$,
i.e., the smallest subgroup or coset of a subgroup containing $A$.

The Freiman--Ruzsa theorem in Finite Torsion Groups~\cite{ruzsa_thm}
explores the relation between these two parameters, in groups of a fixed torsion $r$.
Namely, we are assuming that $r$ is the largest order of an element in the underlying group.

\begin{thm}[Ruzsa~\cite{ruzsa_thm}]\label{freiman-ruzsa}
Let $A$ be a finite subset of an abelian group of torsion $r$.
Then
$$ \frac{|A+A|}{|A|} \leq K \;\;\;\;\; \Rightarrow \;\;\;\;\;
   \frac{|\left\langle A \right\rangle|}{|A|} \leq K^2r^{K^4} .$$
\end{thm}

It is natural to ask how tight this bound is.
To this end, the following function is defined for $r \in \N$ and $K \geq 1$.
$$ F(r,K) = \sup\left\{ \frac{\left|\left\langle A\right\rangle\right|}{|A|} \;\Bigg{|}\;
              A \subseteq \Z_r^n, \; n \in \N, \; \frac{|A+A|}{|A|} \leq  K \right\} .$$
Here and throughout, $\Z_r = \Z/r\Z$.
Note that there is no loss of generality in assuming $A \subseteq \Z_r^n$,
rather than considering a general abelian $r$-torsion group.
Otherwise, $A \subseteq G = \Z_r^n / H$ for some $n$ and $H$,
and the same doubling and spanning constants can be achieved
by taking the preimage of $A$ under the quotient map.

A lower bound on $F(r,K)$ is obtained by taking a set of affinely independent elements.
Specifically, if we choose $A = \{0,e_1,e_2,...,e_{2K-2}\}$,
the canonical basis of $\Z_r^{2K-2}$, for $K \in \frac12 \N$ and $r \geq 3$,
then the doubling constant of $A$ equals $K$, and we have
\begin{equation}\label{lower}
F(r,K) \;\geq\; \frac{r^{2K-2}}{2K-1} .
\end{equation}
This leads to the following conjecture.

\begin{conj}[Ruzsa~\cite{ruzsa_thm}]\label{ruzsa_conj}
There exists some $C \geq 2$ for which $F(r,K) \leq r^{CK}$.
\end{conj}

Green and Ruzsa~\cite{green_ruzsa}
lowered the bound in Theorem~\ref{freiman-ruzsa} to $F(r,K) \leq K^2 r^{2K^2-2}$.
In the special case $r=2$, further progress has been made~\cite
{deshouillers_hennecart_plagne,sanders,green_tao,konyagin,evenzohar}.
In particular, Green and Tao~\cite{green_tao} showed that
$F(2,K) \leq 2^{2K+O(\sqrt{K} \log K)}$, thus settling Conjecture~\ref{ruzsa_conj} for $r=2$.
A refinement of their argument enabled the first author~\cite{evenzohar}
to find the exact value of $F(2,K)$, which turned out to be $\Theta(2^{2K}/K)$.
In this note we extend these techniques to the case of general prime torsion.

\begin{thm}\label{main}
For $p > 2$ prime and $K \geq K_0$,
$$F(p,K) \;\leq\; \frac{p^{2K-2}}{2K-1}.$$
Here $K_0 = 8$ is an absolute constant.
\end{thm}

This verifies Ruzsa's conjecture for prime torsion.
Moreover, the prescribed upper bound is best possible for half-integer $K$,
as was demonstrated in~\eqref{lower}.
Although no attempt was made to optimize $K_0$,
we note that our method is essentially applicable to lower values of $K$.
For more details, see our comments at the end of Section~\ref{proofsection}.

Theorem~\ref{main} is proven in Section~\ref{proofsection}.
The proof elaborates on methods of subset compressions in~$\F_2^n$,
which were first employed in the present context by Green and Tao~\cite{green_tao}.
Although the general scheme of the proof is similar to~\cite{evenzohar},
the transition from $2$ to general $p$ requires a new type of compressions.
The definition and properties of these compressions appear in Section~\ref{compsection}.
We also make use of the following classical result~\cite{cauchy,davenport}.

\begin{thm}[Cauchy--Davenport]\label{c-d}
For non-empty $C,D \subseteq \F_p$, $|C+D| \geq \min(|C|+|D|-1,p)$.
\end{thm}

\section{Compressions in \texorpdfstring{$\F_p^n$}{}}\label{compsection}

Let $e_1,...,e_n$ be the standard basis of $\F_p^n$.
For $u \in \F_p^n$ we denote $u = \sum_{i=1}^{n} u_i e_i$ where $u_i \in \{0,...,p-1\}$.
For $u,v \in \F_p^n$ we say that $u \prec v$ in the \emph{lexicographic order},
if $u_i < v_i$ for the largest coordinate $i$ for which $u_i \neq v_i$.
For $L \subseteq \F_p^n$ and $k \leq |L|$,
we denote by $\IS(k,L)$ the \emph{initial segment} of size $k$ of $L$,
which is the set of the $k$ smallest elements of $L$ in the lexicographic order.

The line passing through $u \in \F_p^n$
in the direction of a nonzero $v \in \F_p^n$ is denoted
$L_v^u = u + \F_p v = \{u+kv \;|\; k \in \F_p\}$,
and the partition of $\F_p^n$ to \emph{$v$-lines} is denoted
$\mathcal{L}_v = \{L_v^u | u \in \F_p^n\}$.
Let $A \subseteq \F_p^n$. The \emph{$v$-compression} of $A$ is defined by
replacing $L \cap A$ with a same-cardinality initial segment of $L$
for every $v$-line $L$:
$$ C_v(A) = \bigcup\limits_{L \in \mathcal{L}_v} \IS(|A \cap L|,L) .$$
Note that $|C_v(A)|=|A|$.
If $C_v(A)=A$, we say that $A$ is \emph{$v$-compressed}.
One can check easily that $C_v(A)$ is $v$-compressed.
For $p = 2$ and $v = e_i$,
the compression operator $C_v$ coincides with $C_{\{i\}}$
as in~\cite{green_tao,evenzohar}.
Compressions along general multidimensional subspaces can be defined analogously,
but are not necessary for this work.

\begin{lemma}\label{sumset}
Let $A,B \subseteq \F_p^n$ and $v \in \F_p^n \setminus \{0\}$.
Then $C_v(A) + C_v(B) \subseteq C_v(A+B)$.
\end{lemma}

\begin{proof}
Let $i \in \{1,...,n\}$ be the largest coordinate such that $v_i \neq 0$.
Without loss of generality $v_i=1$,
otherwise replace $v$ with $v_i^{-1}v$ without affecting $\mathcal{L}_v$.
We first show that for every two $v$-lines $L_v^u, L_v^w$
and $c,d \in \{1,...,p\}$,
\begin{equation}\label{LuLu}
IS\left(c,L_v^u\right) + IS\left(d,L_v^w\right)
\;=\; IS\left(\min(c+d-1,p),L_v^{u+w}\right).
\end{equation}
Note that for each $v$-line $L$,
there exists a unique $u \in \F_p^n$ such that $L = L_v^u$ and $u_i=0$.
Therefore, we can assume without loss of generality $u_i = w_i = 0$,
and hence $(u+w)_i=0$ as well.
For this choice of $u$, the lexicographic order of $L_v^u$ is
$u \prec u+v \prec u+2v \prec ... \prec u+(p-1)v$,
and of course the same holds if we replace $u$ with $w$ or $u+w$.
Now~\eqref{LuLu} reduces to addition of initial segments in~$\F_p$, which can be checked easily.

More generally, if $C,D \subseteq \F_p$ are non-empty such that
$A \cap L_v^u = u + Cv$ and $B \cap L_v^w = w + Dv$,
then $(A \cap L_v^u)+(B\cap L_v^w) = u+w+(C+D)v$.
By the Cauchy--Davenport theorem (Theorem ~\ref{c-d}) applied on $C$ and $D$,
$$ \min(|A \cap L_v^u| + |B \cap L_v^w| - 1, p)
   \;\leq\; |(A\cap L_v^u)+(B\cap L_v^w)| \;\leq\; |(A+B) \cap L_v^{u+w}|.$$
Putting this into~\eqref{LuLu}, we have
$$ IS\left(|A \cap L_v^u|,L_v^u\right) + IS\left(|B \cap L_v^w|,L_v^w\right)
   \;\subseteq\; IS\left(|(A+B) \cap L_v^{u+w}|,L_v^{u+w}\right) .$$
Note that this holds also when one of the summands on the left hand side is empty.
The proof is completed by taking the union over all $u$ and $w$.
\end{proof}

\begin{cor}\label{sumsetAA}
$|C_v(A) + C_v(A)| \leq |A+A|$.
\end{cor}

By Corollary~\ref{sumsetAA}, compressions can reduce $|A+A|$.
However, they might reduce $|\left\langle A \right\rangle|$ as well.
In order to apply compressions most effectively in the proof of Theorem~\ref{main},
we only apply compressions that preserve the affine span.
For $A \subseteq \F_p^n$, we say that $A$ is \emph{$\EEE$-compressed}
if $E=\{0,e_1,e_2,...,e_n\} \subseteq A$
and $A$ is $v$-compressed for every $v$-compression such that $E \subseteq C_v(A)$.

\begin{exm*}
Let $E = \{0,e_1,e_2\} \subseteq \F_5^2$.
Then $A = \{0,e_1,2e_1,3e_1,e_2\}$ is $\EEE$-compressed,
but $A' = A \cup \{2e_2\}$ is not $\EEE$-compressed since $E \subseteq C_v(A') \neq A'$ for $v = e_2-e_1$.
\end{exm*}

We conclude this section with a lemma describing the structure of $\EEE$-compressed subsets.

\begin{lemma}[Structure of Compressed Subsets]\label{structure}
Let $A$ be an $\EEE$-compressed subset of $\F_p^n$.\\
Suppose:
\begin{itemize}
\item
$H = \span\{0,e_1,...,e_h\}$ is the maximal such subgroup contained in $A$,
\item
$a_i = e_{h+i}$ for $i \in \{1,...,m\}$,
where $m = \codim H=n-h$,
\item
$A_i = A \cap (a_i + H)$ for $i \in \{2,...,m\}$,
\item
$A_1 = A \cap (q a_1 + H)$
where $q \in \{1,...,p-1\}$ is maximal such that the intersection is non-empty.
\end{itemize}
Then
$$ A = H \cup (a_1+H) \cup (2a_1+H) \cup ... \cup ((q-1)a_1+H)
   \cup A_1 \cup A_2 \cup ... \cup A_m. $$
\end{lemma}


\begin{proof}
The lemma claims that certain $H$-cosets are contained in $A$ while others are disjoint from~$A$.
Throughout the proof, we pick pairs of elements $u,v \in \F_p^n$ such that $u \prec v$ and $A$ is $(v-u)$-compressed,
and then accordingly reason about their incidence to $A$.

We start with a useful observation regarding $\EEE$-compressed subsets:
If $i \in \{1,...,n\}$ and~$v \in A \cap \span\{e_1,...,e_{i-1}\}$,
then $A$ is $(e_i - v)$-compressed.
Indeed, $E \subseteq C_{e_i - v}(A)$
since every element in $E$ is the lexicographically smallest in its $(e_i - v)$-line,
except for $e_i$ which is preceded only by $v$, and $v \in A$.
In particular, taking $v=0$, $A$ is $e_i$-compressed for $i \in \{1,...,n\}$.
In other words, $A$ is a down-set in the partial order of comparison in all coordinates.

Now, let $H$, $h$, $m$, $a_1,...,a_m$, $A_1,...,A_m$ and $q$
be as in the statement of the theorem.
By the above observation, $A$ has the following properties.
\begin{enumerate}
\item
Every element in $(q-1)a_1+H$ is lexicographically smaller than every element in $qa_1+H$,
and such two elements lie on some $(a_1 - v)$-line where $v \in H$.
Since $A$ is $(a_1 - v)$-compressed for all $v \in H$, and $A$ intersects $qa_1+H$,
this means that $(q-1)a_1+H$ is contained in $A$.
\item
By maximality, $H$ is contained in $A$ but $H + \F_p a_1$ is not.
In other words, there exists $v \in A \cap (H + \F_p a_1)$ such that $v+a_1 \not \in A$.
Now, $A$ is $(a_i - v)$-compressed for $i \in \{2,...,m\}$.
Since $a_1+a_i$ is larger than $v+a_1$ in an $(a_i - v)$-line, $a_1 + a_i \notin A$ too.
\item
$A$ is $(a_i - a_1)$-compressed for $i>1$.
As $a_1+a_i \prec 2a_i$, $2a_i \notin A$ as well.
\item
$A$ is $(a_i - a_j)$-compressed for $i>j>1$.
As $2a_j \prec a_j + a_i$, also $a_j + a_i \notin A$.
\end{enumerate}
Examination of the above, together with use of the down-set property,
yield the desired $H$-cosets structure of $A$.
\end{proof}

\section{An Upper Bound on \texorpdfstring{$F(p,K)$}{}}\label{proofsection}

\begin{proof}(of Theorem~\ref{main})
Suppose $A \subseteq \F_p^n$ such that
$|\left\langle A \right\rangle|/|A| = p^{2K-2} / (2K - 1)$ for some $K \geq K_0$,
where $K_0$ will be determined later.
We have to show that $|A+A|/|A| \geq K$.
Since $p^{2K-2} / (2K - 1)$ is monotone in $K$, the theorem would follow.

Without loss of generality we may assume that $E = \{0,e_1,...,e_n\} \subseteq A$.
Indeed, $|A|$, $|A+A|$ and $|\left\langle A \right\rangle|$
are unaffected by affine transformations.
Let \emph{height}$(a)$ be $a$'s index in the lexicographic order.
Now by induction on $\sum_{a \in A}\height(a)$, $A$ is reduced to be $\EEE$-compressed.
Otherwise apply a $v$-compression such that
$\left\langle C_v(A) \right\rangle=\left\langle E \right\rangle=\left\langle A \right\rangle$,
while $|A+A|/|A| \geq |C_v(A)+C_v(A)|/|C_v(A)| \geq K$
by Corollary~\ref{sumsetAA} and the induction hypothesis.

Therefore, $A$ has the structure described in Lemma~\ref{structure}.
Namely, for $H$, $m$, $a_1,...,a_m$, $A_1,...,A_m$ and $q$ as in the lemma,
$$A = \;\bigcup\limits_{i=0}^{q-1}\; (ia_1+H) \;\cup\;\bigcup\limits_{i=1}^{m}\; A_i \;.$$
This partition yields an estimate for $|A|$,
\begin{equation}
\frac{q}{p^m}  \;\leq\; \frac{|A|}{|\spA|} \;\leq\; \frac{m+q}{p^m} \;.
\label{A}
\end{equation}
Likewise, the implied structure of the sum is
$$A+A = \;\bigcup\limits_{i=0}^{2q-1}\; (ia_1+H) \;\cup\;\;\bigcup\limits_{j=2}^{m}\bigcup\limits_{i=0}^{q-1}\;\;(ia_1+a_j+H)
\;\cup\bigcup\limits_{1 \leq i \leq j \leq m}(A_i+A_j) \;.$$
If $2q < p$ then the union is disjoint.
Otherwise, it includes $p-1$ distinct cosets of the form $ia_1+H$, other than $(2q-p)a_1+H$ which contains $A_1+A_1$.
In either case,
\begin{equation}
|A+A| \;\geq\; \min(2q,p-1)|H|+(m-1)q|H| + \sum\limits_{i\leq j}|A_i+A_j|\;	.
\label{AA}
\end{equation}
We can further simplify,
$$ \sum\limits_{i\leq j}|A_i+A_j| \geq \sum\limits_{i\leq j}\max(|A_i|,|A_j|) \geq
   \sum\limits_{i\leq j}\frac{|A_i|+|A_j|}{2} = \frac{m+1}{2}\sum\limits_{i}|A_i| =
   \frac{m+1}{2}(|A|-q|H|) .$$
We substitute this and $|H| = |\spA| / p^m$ into~\eqref{AA}, to obtain
\begin{equation}\label{L}
|A+A| \;\geq\;
\min\limits_{q,m} \left[ \left( \min(2q,p-1)+\frac{(m-3)q}{2}\right) \frac{|\spA|}{p^m} \;+\;
                         \frac{m+1}{2}|A| \right] ,
\end{equation}
where $1 \leq q < p$ and $m$ is restricted by bounds that follow from~\eqref{A}:
a lower bound $m \geq \log_p\left( q|\spA|/|A| \right)$, and an (implicit) upper bound
$$ F_q(m) \;:=\; \frac{p^m}{m+q} \;\leq\; \frac{|\spA|}{|A|} .$$
As we show below (Claim~\ref{claim1}),
the right-hand side of~\eqref{L} decreases with $m$ (real) for fixed $q$.
It suffices therefore to consider only the largest possible value of $m$,
namely $m = F_q^{-1}\left( |\spA|/|A| \right)$.
After some rearrangement, \eqref{L} becomes
$$ \frac{|A+A|}{|A|} \;\geq\;
   \min\limits_q \; G_q\left(F_q^{-1}\left( \frac{|\spA|}{|A|} \right)\right) ,$$
where
$$ G_q(m) \;:=\; \left( \min(2q,p-1)+\frac{(m-3)q}{2}\right) \frac{1}{m+q} + \frac{m+1}{2}
\;=\; \frac{\binom{m+1}{2} + q(m-1) + \min(2q,p-1)}{m+q} .$$
We also show (Claim~\ref{claim2}) that of the $p-1$ real functions $F_q(m) \mapsto G_q(m)$,
the smallest one corresponds to $q=1$.
The theorem is then established by routine verification of the identity
$F_1(m) = p^{2G_1(m)-2} / (2G_1(m) - 1)$, so that $|A+A|/|A| \geq K$.

We now turn to justify the choice of $m$ and $q$.

\begin{claim}[choosing $m$]\label{claim1}
The right-hand side of~\eqref{L} is a decreasing function of $m$ in the relevant interval.
That is, for any fixed $p>2$ and $1 \leq q<p$,
$ \left( 2\min(2q,p-1)+(m-3)q\right) \frac{|\spA|}{2p^m} + (m+1)\frac{|A|}{2} $
is a decreasing function of $m$ whenever
$|\spA|/|A| = p^{2K-2}/(2K-1) \in \left[p^m/(m+q),p^m/q\right]$ and $K \ge K_0$.
\end{claim}

\begin{proof}
Differentiating with respect to $m$ gives
$\frac{|A|}{2}+\frac{|\spA|}{2p^m}\left(q-((m-3)q+2\min(2q,p-1))\ln p\right)$.
For this expression to be negative when $|A| \leq (m+q)\frac{|\spA|}{p^m}$,
it is sufficient to require
$$ m+q \leq ((m-3)q+2\min(2q,p-1))\ln p - q ,$$
or equivalently
$$ m \;\geq\; m(p,q) := \max
\left(\frac{2q-1}{q \ln p - 1}-1 \;,\; \frac{2q+3-2(p-1)\ln p}{q \ln p - 1}+3 \right) .$$
Since $m \geq \log_p\left( q \cdot |\spA|/|A| \right)$ by assumption,
requiring $|\spA|/|A| \geq \max_q \left(p^{m(p,q)}/q\right) $ would clearly make it happen.
In terms of the condition $|\spA|/|A| = p^{2K-2} / (2K - 1)$ for $K \geq K_0$,
we only have to choose $K_0$ accordingly.
One may verify that $K_0 \geq 8$ ensures
$|\spA|/|A| \geq p^{14}/15 \geq p^{10} \geq p^{m(p,q)}$
for every $p$ and $q$ under consideration.
\end{proof}

\begin{claim}[choosing $q$]\label{claim2}
The function $G_q \circ F_q^{-1}$ is minimal for $q=1$. That is,
$$
G_q \circ F_q^{-1} (|\spA|/|A|) \;\ge\; G_1 \circ F_1^{-1} (|\spA|/|A|),
$$
where $|\spA|/|A| = p^{2K-2}/(2K-1)$ for $K \ge K_0$.
\end{claim}

\begin{proof}
Since $G_q$ and $F_q$ are both increasing functions,
the claim is equivalent to $F_q \circ G_q^{-1}$ being maximal for $q=1$.
Solving the quadratic gives $G_q^{-1}(x+\frac12) = x-q+\sqrt{x^2+g(q)}$
where $g(q) := q^2+3q-2\min(2q,p-1)$.
Note that $g(q)$ is always between $q(q-1)$ and $q(q+1)$.
For $q \geq 2$, we have to confirm the inequality,
$$ F_q\left(G_q^{-1}\left(x+\frac12\right)\right)
\;=\; \frac{p^{x-q+\sqrt{x^2+g(q)}}}{x+\sqrt{x^2+g(q)}}
\;\leq\; \frac{p^{2x-1}}{2x}
\;=\; F_1\left(G_1^{-1}\left(x+\frac12\right)\right)
\;.$$
Indeed, for $x \geq 5/2$ it readily follows from $0 \leq g(q) \leq q(q+1)$.
In our setting $|\spA|/|A| = F_1\left(G_1^{-1}\left(K\right)\right)$,
hence choosing $K_0 \geq 3$ is sufficient.
\end{proof}

This concludes the proof of Theorem~\ref{main}.
\end{proof}

\begin{rmks*} (on the proof)
\begin{enumerate}
\item
It is interesting to note that if we fix $q$,
then the function $F_q \circ G_q^{-1}$ gives a tight upper bound.
This can be seen by setting $A_{q,m} =  \{0,e_1,2e_1,...,qe_1,e_2,e_3,...,e_m\}$.
\item
The proof yields $K_0=8$.
One may wonder whether the statement of Theorem~\ref{main} is true for every $K \geq 1$.
The answer to this question is negative, as demonstrated by
$A_{2,2} = \{0,e_1,2e_1,e_2\} \subseteq \F_3^2$.
For $K=1.75$, this subset shows that $F(3,K) \geq 2.25$,
which is higher than the suggested $3^{2K-2}/(2K-1) \approx 2.08$.
Probably further investigation of the structure of compressed sets
is required to establish sharp upper bounds on $F(p,K)$ in this range
(see~\cite{evenzohar}).
\item
Nevertheless, already the proofs of Claims~\ref{claim1}-\ref{claim2}
can be used to obtain smaller $K_0(p)$ for a given~$p$.
For example, one can take $K_0(3)= 6.72$ and $K_0(5) = 2.30$.
We also state without proof that a closer analysis would enable showing that
$K_0(p) \rightarrow 1$ as $p \rightarrow \infty$.
\end{enumerate}
\end{rmks*}

\section{Discussion}

We established in this work the conjecture of Ruzsa for all groups of prime torsion.
The fact that a lexicographic order can be defined in $\Z_{p^k}$
(see, e.g.,~\cite{bollobas_leader,eliahou_kervaire_plagne}),
such that initial segments minimize cardinalities of sumsets
as in the Cauchy--Davenport theorem,
suggests that these techniques can be extended to prime-power torsion.
Still, a number of technical challenges need to be resolved.

The case of general composite torsion seems more challenging,
as no effective analogs of the compression operators are known in this case.
We note that in some instances, over groups of composite torsion one can find
significantly different extremal structures than over prime or prime-power torsion groups.
For example, in~\cite{grolmusz} explicit Ramsey graphs are constructed, based on
incidence structure over $\Z_6$ which cannot exist over prime-power torsion groups.
Whether this is the case also in the current setting remains to be seen.

\bibliographystyle{abbrv}
\bibliography{freimanff}
\end{document}